\documentclass[11pt]{amsart}

\usepackage{amsthm}
\usepackage{amsmath}
\usepackage{amssymb}
\usepackage{color}

\newtheorem{thm}{Theorem}[section]
\newtheorem{theorem}[thm]{Theorem}

\newtheorem{corollary}[thm]{Corollary}

\newtheorem{notation}[thm]{Notation}

\newtheorem{lemma}[thm]{Lemma}
\newtheorem{proposition}[thm]{Proposition}

\newtheorem{definition}[thm]{Definition}
\theoremstyle{remark}
\newtheorem{remark}[thm]{Remark}

\newcommand{\RR}{\mathbb R}

\newcommand{\CC}{\mathbb C}
\newcommand{\HH}{\mathbb H}

\begin{document}

\title[analog]{associating vectors in $\CC^n$ with rank 2 projections in $\RR^{2n}$:  with applications}
\author[Casazza, Cheng
 ]{Peter G. Casazza and Desai Cheng}
\address{Department of Mathematics, University
of Missouri, Columbia, MO 65211-4100}
\thanks{The authors were supported by 
NSF DMS 1609760; NSF ATD 1321779; ARO W911NF-16-1-0008.
Part of this reseaech was carried out while the authors
 were visiting the Hong Kong University of Science and
 Technology with support from a grant from (ICERM)
 Institute for Computational and Experimental Research in
Mathematics.}

\email{casazzap@missouri.edu, chengdesai@yahoo.com}

\subjclass{42C15}

\begin{abstract}
We will see that vectors in $\CC^n$ have natural analogs as
rank 2 projections in $\RR^{2n}$ and that this association
transfers many vector properties into properties of rank
two projections on $\RR^{2n}$.  We believe that this association
will answer many open problems in $\CC^n$ where the corresponding
problem in $\RR^n$ has already been answered - and vice
versa.
As a application, we will see that phase
retrieval (respectively, phase retrieval by projections) in $\CC^n$ transfers to a variation of phase retrieval
by rank 2 projections (respectively, phase
retrieval by projections) on $\RR^{2n}$.  As a
consequence, we will answer the open problem:  Give the complex
version of Edidin's Theorem \cite{E} which classifies when
projections do phase retrieval in $\RR^n$.
As another application we answer a longstanding open problem
concerning fusion frames by showing that fusion frames
in $\CC^n$ associate with fusion frames in $\RR^{2n}$
with twice the dimension.
 As another application, we
will show that a family of mutually unbiased bases 
in $\CC^n$ has a natural
analog as a family of mutually unbiased rank 2 projections in
$\RR^{2n}$.  The importance here is that there are very few
real mutually unbiased bases but now there are unlimited numbers
of real mutually unbiased rank 2 projections to be used in their 
place.  As another application, we will give a
variaton of Edidin's theorem which gives a surprising
classification of norm retrieval.  Finally, we will show that
equiangular and biangular frames in $\CC^n$ have an analog as equiangular
and biangular rank 2 projections in $\RR^{2n}$. 
\end{abstract}

\maketitle

\section{Introduction}

We will show that vectors in $\CC^n$ have natural analogs as
rank 2 projections in $\RR^{2n}$. 
The strength of this association is that it carries
norm properties with it.
In particular, if $v,w\in \CC^n$ with analogs in $\RR^{2n}$ of
$v',w'$, let
$P_v,P_w$ be the associated rank 2 projections on $\RR^{2n}$.
Then $|\langle v,w\rangle|=\|v\|\|P_vw'\|$ and if $v,w$ are unit vectors then $tr\ P_vP_w =2|\langle
v,w\rangle|^2$.  This allows us to move many properties from
$\CC^n$ to corresponding properties of rank 2 projections in
$\RR^{2n}$.  In particular, families of mutually unbiased bases
(MUBs)
in $\CC^n$ will associate with families of mutually unbiased
rank 2 projections in $\RR^{2n}$.  The importance of this
association is that there are very few mutually unbiased bases
in $\RR^n$.  It is known that in $\CC^n$ there are at most
$n+1$ MUBs and in $\RR^n$ there are at most $\frac{n}{2}+1$ MUBs
\cite{del}.
But these limits are rarely reached.  It is known that the limit
is reached in $\RR^n$ if $n=4^p$ and the limit is reached in
$\CC^n$ if $n=p^t$ where p is a prime ( see \cite{cam} for n=p,
and \cite{woot} for $p^t$). We now see that there are substantially
more cases where $\RR^n$ reaches its maximum for mutually
unbiased rank two projections.  For researchers who
can live with mutually unbiased rank 2 projections instead of
mutually unbiased bases, they now have a large number of possible
families to work with.

This association carries fusion frames in $\CC^n$ to fusion frames
in $\RR^{2n}$ with the same fusion frame bounds.  Using this we
will answer a longstanding problem in fusion frame theory by
showing that for every $m\ge n$, there is a tight fusion frame of
m 2-dimensional subspaces in $\RR^{2n}$.

We will also see that this association transfers equiangular and
biangular tight frames into equiangular and biangular tight
fusion frames.

This association will also carry with it the notion of phase
retrieval in $\CC^n$.
In certain engineering applications, the phase of a signal is lost
during collection and processing.  This gave rise to a need for methods to recover
the phase of a signal.  Phase retrieval in engineering is over
100 years old and has application to a large number of areas  
including speech recognition~\cite{BeRi99,RaJu93,ReBlScCa004}, and applications such as X-ray crystallography~\cite{BaMn86,Fi78,Fi82}.   
The concept of \textit{phase retrieval} for Hilbert space
frames was introduced in 2006 by Balan, Casazza, and Edidin~\cite{C} and since then it has become an active area of research.

Phase retrieval has been defined for projections as well as for vectors. \textit{Phase retrieval by projections} occur in real life problems, such as crystal twinning~\cite{Dr010}, where the signal is projected onto some lower dimensional subspaces and has to be recovered from the norms of the projections of the vectors onto the subspaces. We refer the reader to~\cite{CCPW} for a detailed study of phase retrieval by projections.  Phase retrieval has been significantly generalized in \cite{WX}.  

A fundamental
result concerning phase retrieval by projections
due to Edidin \cite{E} is that a family
of projections $\{P_i\}_{i=1}^m$ does phase retrieval in $\RR^n$
if and only if for all $0\not= x \in \RR^n$, the vectors
$\{P_ix\}_{i=1}^m$ span $\RR^n$.  It has been an open question
whether there is a complex analog to this theorem.  We will answer
this question in this paper.  First, we will show that complex
phase retrieval by vectors in $\CC^n$ is equivalent to a problem of real phase retrieval by rank two projections in $\RR^{2n}$.
Next, we will give a new {\it geometric} proof of this
theorem.  Finally, we combine these two techniques to give the
complex analog of Edidin's Theorem.

Finally, we will give a classification of norm
retrieval which is an analog of Edidin's theorem.

\section{Preliminaries}

In this section we will introduce the concepts which will be
used throughout the paper.  For notation, we write $\HH^n$ for
a real or complex n-dimensional Euclidean space.  If we need
to restrict ourselves to one of these choices, we will write
$\RR^n$ or $\CC^n$.

\vskip12pt
\noindent {\bf Note:}  We will rely heavily on the fact that given
$x,y \in \RR^n$, $x-y \perp x+y$ if and only if $\|x\|=\|y\|$.

\begin{definition}
A family of vectors $\{v_i\}_{i=1}^m$ is a {\bf frame} for
$\HH^n$ if there are constants $0<A\le B<\infty$ so that
\[ A\|x\|^2 \le \sum_{i=1}^m|\langle x,v_i\rangle|^2 
\le B\|x\|^2,\mbox{ for all }x\in \HH^n.\]
If $A=B$ this is a {\bf tight frame} and if $A=B=1$ it is
a {\bf Parseval frame}.
\end{definition}

We will be working with phase retrieval and phase retrieval by
projections.

\begin{definition}
A family of vectors $\{v_i\}_{i=1}^m$ does {\bf phase retrieval} on
$\HH^n$ if whenever $x,y\in \HH^n$ satisfy
\[ |\langle x,v_i\rangle|=|\langle y,v_i\rangle|,
\mbox{ for all }i=1,2,\ldots,m,\]
we have that $x=cy$ for some $|c|=1$.

We say a family of subspaces $\{S_i\}_{i=1}^m$ (or their orthogonal
projections $\{P_i\}_{i=1}^m$) on $\HH^n$ do {\bf phase retrieval} 
if whenever $x,y\in \HH^n$ satisfy
\[ \|P_ix\|=\|P_iy\|,\mbox{ for all }i=1,2,\ldots,m,\]
then $x=cy$ for some $|c|=1$.
\end{definition}

\begin{remark}
We will find it convenient to work with the contrapositive.  I.e.
$\{P_i\}_{i=1}^m$ does phase retrieval if and only if whenever
$x \not= cy$ for any $|c|=1$, there is a $1\le j \le m$ so that
$\|P_jx\|\not= \|P_jy\|$. Also, whenever $\|P_jx\|\not= \|P_jy\|$,
for some j,
we say that $\{P_j\}_{j=1}^m$ {\bf distinguishes between $x,y$}.
\end{remark}

We now have:

\begin{lemma}
Given a set of projections $\{P_j\}_{j=1}^m$ on $\HH^n$, the set of vectors they
cannot distinguish are those $v,w$ for which $v-w\perp P_j(v+w)$.
\end{lemma}

\begin{proof}
We compute:
\[
\langle v-w,P_j(v+w)\rangle = \langle P_jv-P_jw,P_jv+P_jw\rangle
= \|P_jv\|^2-\|P_jw\|^2.\]
So $v-w \perp P_j(v+w)$ if and only if $\|P_jv\|=\|P_jw\|$.
\end{proof}

We will need the complement property for families of vectors.

\begin{definition}
A family of vectors $\{v_i\}_{i=1}^m$ in $\HH^n$ has the 
{\bf complement property} if whenever $I\subset \{1,2,\ldots,m\}$,
either span\ $\{v_i\}_{i\in I}= \HH^n$ or span $\{v_i\}_{i\in I^c}
= \HH^n.$
\end{definition}

A fundamental result in this area \cite{C} is:

\begin{theorem}
If vectors $\{v_i\}_{i=1}^m$ do phase retrieval in $\HH^n$ then
they have the complement property.
\end{theorem}

It was also shown in \cite{C} that a family of vectors with
complement property in $\RR^n$ does phase retrieval but this
implication does not hold in $\CC^n$.  A family of vectors 
$\{v_i\}_{i-=1}^m$ is {\bf full spark} if for every $I\subset
\{1,2,\ldots,m\}$ with $|I|=n$, span $\{v_i\}_{i\in I}=\HH^n$.
So a full spark family with $m\ge 2n-1$ has the complement
property.

Edidin \cite{E} gave a fundamental classification of phase retrieval 
by projections for
$\RR^n$ in terms of the spans of $\{P_ix\}_{i=1}^m$, for 
$x\in \RR^n$.

\begin{theorem}[Edidin]\label{T1}
Let $\{S_i\}_{i=1}^m$ (with respective projections $\{P_i\}_{i=1}^m$)
be subspaces of $\RR^n$.  The following are equivalent:
\begin{enumerate}
\item $\{P_i\}_{i=1}^m$ does phase retrieval.
\item For every $0\not= x \in \RR^n$, span $\{P_ix\}_{i=1}^m =
\RR^n$.
\end{enumerate}
\end{theorem}

The corresponding result for frames (I.e. rank one projections)
has been done in \cite{B}.  The necessity of the condition for
frames also appeared in \cite{dustin}.
For $\CC^n$, (2) does not imply (1) in the theorem in general.
For example, if $\{v_i\}_{i=1}^5$ is a full spark family of
vectors in $\CC^3$, then it has complement property and so
if $P_i$ is the projection onto span $v_i$ then for every
$0\not= x \in \CC^3$, span $\{P_ix\}_{i=1}^5 = \CC^3$.  But any
family doing phase retrieval in $\CC^3$ must contain at least
8 vectors \cite{E}.  However, (1) does imply (2) in $\CC^n$
\cite{dustin}.

\begin{theorem}\label{tt7}
If projections $\{P_i\}_{i=1}^m$ do phase retrieval on $\HH^n$
then for every $0\not= x \in \HH^n$, span $\{P_ix\}_{i=1}^m =
\HH^n$.
\end{theorem}

\begin{proof}
We will prove the contrapositive.  So assume there is an $0\not=
x \in \HH^n$ so that span $\{P_ix\}_{i=1}^m \not= \HH^n$.  Choose
$0\not= y \in \HH^n$ so that $y\perp P_ix$ for all $i=1,2,\ldots,m.$
Then
\[ \langle P_ix, y\rangle = \langle P_ix,P_iy\rangle =0,
\mbox{ for all }i=1,2,\ldots,m.\]
Let $w=x+y$ and $v=x-y$.  Then for all $i$,
\[ \|P_iw\|^2 = \|P_ix+P_iy\|^2 = \|P_ix\|^2+\|P_iy\|^2
= \|P_ix-P_iy\|^2 = \|P_iv\|^2.\]
But, $w \not= cv$ for any $|c|=1$.  Since if $w=cv$ then either
$c=\pm 1$ and so either x=0 or y=0, or $x=\frac{1+c}{1-c}y$ which
combined with $P_ix \perp P_iy$ implies again x=0.  I.e.  $\{P_i\}_{i=1}^m$ fails to do phase retrieval.
\end{proof}

\section{A New Proof of Edidin's Theorem}

In this section we establish a new proof of Theorem \ref{T1}.
This proof is {\bf geometric} and it will allow us to generalize
it to the complex case.  It will also direct us to introduce a
natural association between vectors in $\CC^n$ and two dimensional
subspaces of $\RR^{2n}$.

To give our geometric proof of Theorem \ref{T1}, we will need
a sequence of results.

\begin{lemma}
Given a subspace S of dimension d in $\RR^n$, for any point x, either S is in the orthogonal complement of x or there exists a subspace 
$S'$ of dimension d-1 contained in S that is also contained in the orthogonal complement of x.
\end{lemma}

\begin{proof}
Let $C$ be the orthogonal complement of $x$ in $\RR^n$. 
Then $C$ is a hyperplane so its intersection with $S$ is either all of $S$ or a subspace of dimension $d-1$.
\end{proof}

From the above it is clear that for any point $x$ either $S$ is orthogonal to $x$ or there exists an orthonormal basis $b_1,...b_d$ of $S$ such that $b_1$ is not orthogonal to $x$ but the rest of the elements in the orthonormal basis are orthogonal to $x$.

\begin{remark}Given two real numbers $a$ and $b$, $|a| = |b|$ if and only if $a = b$ or $a = -b$. Hence given two vectors $v$ and $w$, 
$|<x,v>| = |<x,w>|$ if and only if $x$ is orthogonal to $v-w$ or $v+w$.
\end{remark}
Another simple observation we will heavily use is:

\begin{lemma}
Given two vectors $x,y \in \HH^n$, letting $v=\frac{x-y}{2}$ and $w=\frac{x+y}{2}$, we have that $w+v=x$ and $w-v=y$.  
\end{lemma}

  The
next theorem is a variation of the argument of Theorem \ref{t1}.

\begin{theorem}\label{t2}
Given a family of subspaces $\{S_{i}\}_{i=1}^m$ of $\RR^n$ 
with respective projections $\{P_i\}_{i=1}^m$, and a point $0\not= x \in \RR^n$, let M = $span \ \{P_ix\}_{i=1}^m$. For any $y \in M^\perp$, we have for all $1\leq i \leq m$, $\|P_{i}(v)\| = \|P_{i}(w)\|$ for all $w$ and $v$ such that $w+v = x$ and $w-v = y$.
\end{theorem}

\begin{proof}
Given any $S_i$, if $x=w+v\perp S_i$ then $P_iw=-P_iv$ and so $\|P_{i}(w)\| = \|P_{i}(v)\|$.
If $S_i$ is not orthogonal to $x$ then as mentioned above there exists an orthonormal basis $b_1,...b_d$ of $S_i$ such that $b_1$ is not orthogonal to $x$ but the rest of the vectors are
orthogonal to x. Hence the projection of $x$ onto $S_i$, $P_ix$, is a nonzero scalar multiple of $b_1$, 
 $<w,b_1> \neq -<v,b_1>$ and for all $2 \leq i \leq d$, $<w,b_i> = -<v,b_i>$. 
 Hence $\|P_{i}(v)\| = \|P_{i}(w)\|$ if and only if $<w,b_1> = <v,b_1>$. But, this happens only if $w-v$ is orthogonal to $b_1$.

Considering the above argument over all $i$ we see if we pick any point $y \in M^\perp$ such that $w+v = x$ and $w-v = y$, then for all $1\leq i \leq m$, $\|P_{i}(v)\| = \|P_{i}(w)\|$.
\end{proof}

\begin{theorem}\label{t3}
Given the assumptions of Theorem \ref{t2},
given $x,y \in \RR^n$, 
if we can pick $w$ and $v$ such that $w+v = x$ and $w-v=y$ are not orthogonal to $M$ then for some $i$,
$\|P_{i}(v)\| \neq \|P_{i}(w)\|$
\end{theorem}
\begin{proof}
Since $M$ was the span of all the $\{P_ix\}_{i=1}^m$ then for some $i$, $w-v$ is not orthogonal to $P_{i}(x)$. Hence for some orthonormal basis $b_1,...b_d$ of $S_i$ we have $<w,b_1> \neq <v,b_1>$ and $<w,b_1> \neq -<v,b_1>$ (by construction) and for all other $i$, $<w,b_1> = -<v,b_1>$. Clearly $\|P_{i}(v)\| \neq \|P_{i}(w)\|$ by the Pythagorean theorem.
\end{proof}

We now prove Edidin's Theorem \cite{E}.

\begin{theorem}
If a set of subspaces $\{S_{i}\}_{i=1}^m$ of $\RR^n$ with
respective projections $\{P_i\}_{i=1}^m$ does not do real phase retrieval then for some $0\not= x$, $span\ \{P_i(x)\}_{i=1}^m
\not= \RR^n$.
\end{theorem}

\begin{proof}
Suppose there exists $v$ and $w$, $v \neq \pm w$ such that for all $1\leq i \leq m$, $\|P_{i}v\| = \|P_{i}w\|$. Clearly 
$w+v$ and $w-v$ are nonzero. By the above theorem choose $w+v$ to be $x$. Clearly $w-v$ must be orthogonal to $span\ \{P_ix\}_{i=1}^m$ by Theorem \ref{t3}.
\end{proof}

The other direction of Theorem \ref{T1} is Theorem \ref{tt7}.

\section{Turning vectors in $\CC^n$ into rank 2 projections
on $\RR^{2n}$}

We need a piece of notation.

\begin{notation}
For the rest of the paper, for a complex vector $v = (a_1 + ib_1,...a_n + ib_n)$ in $\CC^n$ let $v' = (a_1,b_1,...,a_n,b_n)$ and $v'' = (-b_1,a_1,...,-b_n,a_n)$ be vectors in $\RR^{2n}$.  We will
write $S_v=span\ \{v',v''\}$ and $P_v$ as the rank 2 projection
of $\RR^{2n}$ onto $S_v$.
\end{notation}

The following are immediate from the definition.

\begin{proposition}\label{p1}
Given vectors $v=(a_1+ib_1,\ldots,a_n+ib_n$ and
$w=(a_1'+ib_1',\ldots,a_n'+b_n')$.  The following hold:
\begin{enumerate}
\item we have 
\[ v'\perp v''\mbox{ and } \|v\|^2=\|v'\|^2.\]
\item We have
\[ \langle w',v'\rangle = \sum_{j=1}^n(a_ja_j'+b_jb_j')
\mbox{ and }\langle w',v''\rangle =
\sum_{j=1}^n (a_jb_j'-a_j'b_j).\]
\item We have
\[ \langle w'',v''\rangle = \langle w',v'\rangle.\]
\item We have
\[ \langle w,v\rangle = \langle w',v'\rangle +i \langle w',v''
\rangle.\]
\item It follows that if $\|v\|=1$ then 
\[ |\langle v,w\rangle| = |\langle w',v'\rangle+i\langle w',v''\rangle|=\sqrt{|\langle w',v'\rangle|^2+|\langle w',
v''\rangle|^2}= \|P_vw'\|.\]
\item If $\|v\|\not= 1$ then
\[ |\langle v,w\rangle|=\|v\|\|P_vw'\|.\]
\item $(iw)'= w''$.
\item  Given $z=(a_1,b_1a_2,b_2,\ldots,a_n,b_n)\in \RR^{2n}$,
$z=v'$ where $v=(a_1+ib_1,\ldots,a_n+ib_n)\in \CC^n$. 
Hence, $\RR^{2n}=\{v':v\in \CC^n\}$.
\end{enumerate}
\end{proposition}

\begin{corollary}\label{cor}
If $\{v_j\}_{j=1}^m$ is an orthonormal set of vectors in $\CC^n$
then $\{v_j',v_j''\}_{j=1}^m$ is an orthonormal set of vectors in
$\RR^{2n}$.  
\end{corollary}

\begin{proof}
This is immediate by (1) and (4) in Proposition \ref{p1}.
\end{proof}

As a consequence:

\begin{corollary}\label{c11}
Given complex vectors $v,w_1,w_2$ let P be the projection
onto $S_v$.  The following are equivalent:
\begin{enumerate}
\item $|<w_1,v>| = |<w_2,v>|$.
\item $\|Pw_1'\|=\|Pw_2'\|.$
\end{enumerate}  
 In particular if $w_1 = cw_2$ where $c\in \CC^n$ and $|c|=1$ then $\|Pw_1'\|=\|Pw_2'\|.$
\end{corollary}

\begin{lemma}
The rotation (multiplication by a unit complex scalar $cos\ \theta
+ i sin\ \theta$) of a complex vector $w = (a'_1 + ib'_1,...a'_n + ib'_n)$ is the same as taking $\cos(\theta)w$ + $\sin(\theta)iw$.

Hence 
\[(\cos(\theta)w + \sin(\theta)iw)' = \cos(\theta)w' +  \sin(\theta)w''.\]
\end{lemma}

\begin{remark}
It follows that the vectors obtained by multiplying $w$ by any unit norm complex scalar would associate with the points on the circle of radius $\|w\|$ in $S_w$
\end{remark}

\begin{theorem}\label{t20}
Given any nonzero vector $v$, and complex scalar 
$c=a+bi\not= 0$, $[(a+ib)v]' =av'+bv''$ and we have that
$S_v=S_{cv}$.
\end{theorem}

\begin{proof}
Let $V = span\ \{(cv)',(cv)''\}$.  We will show that $S_v=V$.
We first show $V$ is contained in $S_v$
We know $v$ identifies with $v'$ and $iv$ identifies with $v''$ hence both $v'$ and $(iv)'$ are in $S_v$. 
Next, given any vector of the form $av' + bv''$ we will
show that $av + ibv $ identifies with this vector.  
Clearly $av + ibv = a(a_1 + ib_1,...a_n + ib_n) + 
b(-b_1+ia_1,...-b_n+ia_n)$ which identifies with $av' + bv'' = a(a_1,b_1,...a_n,b_n) + b(-b_1,a_1,...-b_n,a_n)$
Hence $V$ is contained in $S_v$

Now given a complex scalar $(a+bi)$. As shown above $(a+bi)v$ identifies with $av' + bv''$ which is a vector in $V$.
Hence $S_v$ is contained in $V$.
It follows that  
$S_v=V$.
\end{proof}

We may define an equivalence relation on $\CC^n \setminus 0$ by
saying two vectors are {\it equivalent} if and only if one is a complex scalar multiple of the other.
It is clear that two vectors $v$ and $w$ are in the same equivalence class if and only if $S_v = S_w$.

\begin{theorem}\label{t11}
For any two subspaces $S_v$ and $S_w$, if $S_v \bigcap S_w \not= \{0\}$ then $S_v=S_w$.  Moreover, if $v \perp w$ then $S_v \perp S_w$.
\end{theorem}
\begin{proof}
Assume there is a $0\not= x$ with $x' \in S_v \bigcap S_w \not= 0$. Then we have that $x = av$ and $x = bw$ for some nonzero complex scalars $a$ and $b$ and hence $v = a^{-1}bw$. From 
the above $v$ and $w$ are in the same equivalence class and
so $S_v = S_w$.

The moreover part follows from Corollary \ref{cor}.
\end{proof}

Putting this altogether,

\begin{theorem}
Given a set of complex vectors $\{v_i\}_{i=1}^{m} \in \CC^n$ which do phase retrieval in $\CC^n$, if we identify $\CC^n$ with $\RR^{2n}$, for any vector $w$, the 
only points that cannot be distinguished from $w'$ by projections onto $S_{v_i}$ are any points on the circle of radius $\|w\|$ in the subspace $S_w$ spanned by $w'$ and $w''$.
\end{theorem}

\section{Complex Phase Retrieval and Rank Two Projections}

If we let $\{P_i\}_{i=1}^m$ be the projections onto $S_{v_i}$,
we see for any nonzero point $x$ and any $c \in \RR \setminus {0}$, $span\ \{P_ix\} = span\ \{P_i(cx)\}_{i=1}^m$. Hence when looking at the span of the projections we will identify a point with any scalar multiple of it. We see then that $y$ is in the orthogonal complement of the span if and only if $cy$ is in the orthogonal complement for any $c \in \RR \setminus {0}$. Hence when we look at vectors in the orthogonal complement we will identify a vector $y$ with any scalar multiple of it.

\begin{theorem}
Given a vector v $\in \CC^n$, let w = iv and $m = (\cos(\frac{\pi}{4}) + i\sin(\frac{\pi}{4}))v$. Then $v+w = \sqrt{2}m,$ and 
$ v'+w' = \sqrt{2}m'$.
\end{theorem}

\begin{proof}
$v+w = (i+1)v = \sqrt{2}(\cos(\frac{\pi}{4}) + i\sin(\frac{\pi}{4}))v = \sqrt{2}m$
\end{proof}

\begin{corollary}\label{cc2}
For every nonzero $m' \in \RR^{2n}$ there exists nonzero $v'$ and $w'$ such that $w = iv$ (hence $w' \neq v'$) and $v' + w' = m'$
\end{corollary}

\begin{proof}
This follows from the theorem above by taking $v = \dfrac{\cos(\frac{\pi}{4}) - i\sin(\frac{\pi}{4})}{\sqrt{2}}m$.
\end{proof}

\begin{proposition}
Given a vector $v\in \CC^n$ with projection $P$ onto $S_{v}$,
for every $w' \in \RR^{2n}$, $Pw' \perp w''$.
\end{proposition}

\begin{proof}
By Corollary \ref{cc2}, there exist $0\not= x,y\in \CC^n$ so that
$x=iy$ and $x'+y'=w'$.  Hence, $\|x'\|=\|y'\|$ and $|\langle x,v\rangle|
=|\langle y,v\rangle|$.  So $y'\in S_x$.  Hence, $w'\in S_x$ and
so by Theorem \ref{t11}, $S_w=S_x$.  Now, 
\[ \langle x'-y',x'+y'\rangle =\|x'\|-\|y'\| =0
\mbox{ and } x'-y' \in S_w.\]
But, $w''$ (and its multiples) are the only vector in $S_w$ orthogonal to $w'$.  So $w''=x'-y' \perp P(x'+y')=Pw'$.   
\end{proof}

\begin{corollary}
Given vectors $\{v_j\}_{j=1}^m$ in $\CC^n$ with projections 
$\{P_j\}_{j=1}^m$ onto $\{S_{v_j}\}_{j=1}^m$ and
given any nonzero $w' \in \RR^{2n}$ let $M = span\ \{P_jw'\}_{j=1}^m$. Then $w'' \perp M$.
\end{corollary}

\begin{theorem}\label{t12}
Given vectors $\{v_j\}_{j=1}^m$ in $\CC^n$ which do phase
retrieval with projections 
$\{P_j\}_{j=1}^m$ onto $\{S_{v_j}\}_{j=1}^m$ and
given any nonzero $w' \in \RR^{2n}$ let $M = span\ \{P_jw'\}_{j=1}^m$. Then $M^{\perp}=span \{w''\}$.
\end{theorem}

\begin{proof}
Given $v' \in M^{\perp}$, there exists
 two vectors $x'$ and $y'$ with $x'+y'=w'$ and $x'-y'=v'$.
 It follows that
 \[ 0 = \langle P_j(x'+y'),x'-y'\rangle = 
 \langle P_j(x'+y'),P_j(x'-y')\rangle = \|P_jx'\|^2-\|P_jy'\|^2.\]
 So $\|P_jx'\|=\|P_jy'\|$ and
 by Corollary \ref{c11}, we have that $|\langle x',v_j\rangle|
 =|\langle y',v_j\rangle|$ for all $j=1,2,\ldots,m$.  Therefore,
 since we have phase retrieval,
 $x=cy$ for some $|c|=1$ and hence
 $y' \in S_x$ and $\|y'\| = \|x'\| \neq 0$. Since $S_x$ is a subspace $x' + y' = w',x'-y' \in S_x$.  By Theorem \ref{t11}
$S_w=S_x$.  But, $\{w',w''\}$ is an orthonormal basis for
$S_w$ and $x'+y'=w'\perp x'-y'$.  It follows that $x'-y'$ is a 
multiple of $w''$. Hence, $w'' \perp M$ and its scalar multiples
are the only vectors orthogonal to M.
\end{proof}

Now we will see how to reformulate complex phase retrieval by
vectors in $\CC^n$ into a variant of real phase retrieval for our class of
rank two projections on $\RR^{2n}$.

\begin{theorem}\label{t21}
Given a set  of complex vectors $\{v_j\}_{j=1}^{m} \in \CC^n$
let $\{P_j\}_{j=1}^m$ be the corresponding projections onto
$S_{v_j}$. The following are equivalent:
\begin{enumerate}
\item $\{v_j\}_{j=1}^n$ does phase retrieval.
\item
For each point $x' \in \RR^{2n}$, $x''$ is the only vector 
in $\RR^{2n}$ orthogonal to $M=span\ \{P_ix'\}_{i=1}^m$.
\item  For each point $x' \in \RR^{2n}$, $M=span\ \{P_ix'\}_{i=1}^m$
is a hyperplane.
\end{enumerate}
\end{theorem}

\begin{proof}
$(1)\Rightarrow (2)$:  This is Theorem \ref{t12}.

$(2) \Rightarrow (1)$: 
Choose vectors $v,w$ with 
\begin{equation}\label{E5}
|\langle v,v_i\rangle|=|\langle w,v_i\rangle|
\end{equation} for all
$i=1,2,\ldots,m$.  We first need to observe that
$\langle v,v_i\rangle \not=0$ for some i.  For
otherwise we would have that $v'\in M^{\perp}$ and
since $v' \perp v'' \in M^{\perp}$, this contradicts
our assumption that $dim\ M^{\perp}=1$.
 By Corollary \ref{c11}, $\|P_iv'\|=\|P_iw'\|$.
It follows that
\[ \langle P_i(v'+w'),v'-w'\rangle=
\langle P_i(v'+w' ),P_i(v'-w')\rangle = \|P_iv'\|^2-\|P_iw'\|^2
=0.\]
 Hence, $v'-w' \perp P_ix'=P_i(v'+w')$ for all $i=1,2,\ldots,m$.  By assumption
 (2), $v'-w'=cw''\in S_w$ and so $v'\in S_w$.  By Theorem \ref{t11},
 $S_v=S_w$ and so $v=cw$ for some c.  Now by Equation \ref{E5},
 $|c|=1$.
 
 $(2)\Leftrightarrow (3)$:  This is clear.
\end{proof}

\begin{corollary}
Let $\{v_i\}_{i=1}^m$ do phase retrieval on $\CC^n$ and for
$x\in \CC^n$ let $I=\{i:\langle x,v_i\rangle \not= 0\}$. Then
\begin{enumerate}
\item $|I|\ge 2n-1$.
\item $span \ \{v_i\}_{i\in I}=\CC^n$.
\end{enumerate}
\end{corollary}

\begin{proof}
(1)  We proceed by way of contradiction.  Let $P_i$ be the rank
two projection in $\RR^{2n}$ onto $S_{v_i}$.  
If $|I|\le 2n-2$, then for $i\in I^c$, $\|P_ix'\|=|\langle x,v_i\rangle|=0$.  It follows that $|\{i:P_ix' \not= 0\}|\le 2n-2$ and
hence $\{P_ix'\}_{i=1}^m$ does not span a hyperplane in $\RR^{2n}$ contradicting
Theorem \ref{t21}.

(2)   Let $P_i$ be the rank
two projection in $\RR^{2n}$ onto $S_{v_i}$. 
We proceed by way of contradiction.
So assume there exists
$y\in \CC^n$ with $y\perp v_i$ for all $i\in I$.  Then,
$P_iy'=0$ for all $i\in I$.  Since $y''=(iy)'$, it follows
that $P_iy''=0$ for all $i\in I$.  Also, $P_ix=0$ for all $i\in I^c$.
So
\[ M=span\ \{P_ix'\}_{i=1}^m=span\ \{P_ix'\}_{i\in I}.\]
It follows that $y',y'' \perp M$ and so
M is not a hyperplane in $\RR^{2n}$ contradicting Theorem
\ref{t21}. 

\end{proof} 

\begin{corollary}
If $\{v_i\}_{i=1}^{4n-4}$ does phase retrieval in $\CC^n$ and
$I\subset [4n-4]$ with $|I|=2n-2$, then $\{v_i\}_{i\in I}$
and $\{v_i\}_{i\in I^c}$ both span
$\CC^n$.
\end{corollary}

To prove the complex analog of Theorem \ref{T1}, we need
a result.

\begin{proposition}\label{p5}
Let $W$ be a d-dimensional subspace of $\CC^n$.  Then there is a
2d-dimensional subspace $V$ of $\RR^{2n}$ so that for every
orthonormal basis $\{v_i\}_{i=1}^d$ for $W$ we have
$span\ \{S_{v_i}\}_{i=1}^d = V$.  Recall that the $\{S_{v_i}\}_{i=1}^m$ is an
orthogonal set of 2-dimensional subspaces of $\RR^{2n}$.
\end{proposition}

\begin{proof}
Choose any orthonormal basis $\{v_i\}_{i=1}^d$ for
$W$ and let $V=span\ \{S_{v_i}\}_{i=1}^m$.
We want to show that for any other orthonormal basis
$\{w_i\}_{i=1}^m$ for $W$, we have $span\ \{S_{w_i}\}_{i=1}^m = V$.
It suffices to show that for any $w \in W$, $w',w''
\in V$.  If we write 
$w=\sum_{j=1}^d(a+bi)v_i$, by Theorem \ref{t21} we have
\[ w'=\sum_{j=1}^d[(a_j+ib_j)v_j]' = 
\sum_{j=1}^m [a_jv_j'+b_jv_j''] \in V.\]
Since $w'' = (iw)'$, the same argument shows that
$w'' \in V$.    
\end{proof}

We will need one more preliminary result.

\begin{proposition}\label{p7}
Let $W$ be a subspace of $\CC^n$of dimension d, let $\{v_i\}_{i=1}^d$
be an orthonormal basis of $W$ and let $V=span\ \{S_{v_i}\}_{i=1}^d$
be the induced subspace in $\RR^{2n}$.  Let $Q$ (respectively $P$)
be the projection onto $W$ (respectively $V$).  Then for all
$x\in \CC^n$, $\|Qx\|=\|Px'\|$.
\end{proposition}

\begin{proof}
Let $P_i$ be the projections onto $S_{v_i}$.
For any $i$, $|\langle x,v_i\rangle|=\|P_ix'\|$ and so
\[
\|Qx\|^2= \sum_{i=1}^d |\langle x,v_i\rangle|^2
= \sum_{i=1}^d\|P_ix'\|^2
= \|Px'\|^2.\]
\end{proof}

Now we give the complex analog of Edidin's Theorem \cite{E}.

\begin{theorem}
Let $\{W_i\}_{i=1}^m$ be subspaces of $\CC^n$ 
with projections $\{Q_i\}_{i=1}^m$ and let $\{V_i\}_{i=1}^m$ be the corresponding subspaces of $\RR^{2n}$ given in Proposition
\ref{p5} with projections $\{P_i\}_{i=1}^m$.  The following are equivalent:
\begin{enumerate}
\item $\{Q_i\}_{i=1}^m$ does phase retrieval.
\item For every $w'\in \RR^{2n}$, if $M=span\ \{P_iw'\}_{i=1}^m$
then $M^{\perp}= span\ \{w''\}$.
\end{enumerate}
\end{theorem}

\begin{proof}
$(1)\Rightarrow (2)$:  Given $v'\in M^{\perp}$, there exist two
vectors $x',y'$ so that $x'+y'=w'$ and $x'-y'=v'$.  Now,
\[ 0=\langle P_i(x'+y'),x'-y'\rangle = \langle P_ix'+P_iy',
P_ix'-P_iy'\rangle = \|P_ix'\|^2-\|P_iy'\|^2.\]
By Proposition \ref{p7}, $\|Q_ix\|=\|Q_iy\|$ for all $i=1,2,\ldots,
m$.  Since $\{Q_i\}_{i=1}^m$ does phase retrieval, we have that
$x=cy$ for some $|c|=1$ and $\|x'\|=\|y'\|$.  Since $S_x$ is a subspace $x' + y' = w',x'-y' \in S_x$.  By Theorem \ref{t11}
$S_w=S_x$.  But, $\{w',w''\}$ is an orthonormal basis for
$S_w$ and $x'+y'=w'\perp x'-y'$.  It follows that $x'-y'$ is a 
multiple of $w''$. Hence, $w'' \perp M$ and its scalar multiples
are the only vectors orthogonal to M.  

$(2)\Rightarrow (1)$:  Assume $v,w \in \CC^n$ and
$\|Q_iv\|=\|Q_iw\|$ for all $i=1,2,\ldots,m.$  By 
Proposition \ref{p7}, $\|P_iv'\|=\|P_iw'\|$ for 
all i.  Now,
\[ \langle P_i(v'+w'),v'-w'\rangle = \langle P_iv'+P_iw',
P_iv'-P_iw'\rangle = \|P_iv'\|^2-\|P_iw'\|^2=0.\]
By our assumption in (2), $v'-w'=c(v''+w'')$.  It
follows that $S_v=S_w$ and $v'-cv''=w'+cw''$. Since
$\|v'\|=\|v''\|$, $\|w'\|=\|w''\|$, $v'\perp v''$,
$w'\perp w''$, and $v'-cv''=w'+cw'$, a little geometry
shows that $|c|=1$ and so $v=dw$ for some $|d|=1$. 
\end{proof}
 
\section{Mutually Unbiased Bases}

In this section, we will see that for a family of mutually unbiased
bases in $\CC^n$ their corresponding rank 2 projections in 
$\RR^{2n}$ are mutually unbiased. 

\begin{definition}
Two orthonormal bases $\{e_i\}_{i=1}^n,\ \{e_i'\}_{i=1}^n$ for
$\HH^n$ are said to be {\bf mutually unbiased} if 
\[ |\langle e_i,e_j'\rangle|^2 = \frac{1}{n},\mbox{ for all }
i,j=1,2,\ldots,n.\]
\end{definition}

We first need a lemma.

\begin{proposition}\label{p23}
Given $v,w$ be unit vectors in $\CC^n$ let $P,Q$ be the rank
2 projections in $\RR^{2n}$ onto $S_v,S_w$.
Then, $tr\ (PQ) = 2|\langle v,w\rangle|^2$.
\end{proposition}

\begin{proof}
We note that $P=v'v'^*+v''v''^*$ and $Q=w'w'^*+w''w''^*$.
Now we compute using Proposition \ref{p1} (5):  
\begin{align*}
\langle P,Q\rangle &= tr\ (PQ)\\
&= tr\ (v'v'^*+v''v''^*)(w'w'^*+w''w''^*)\\
&= tr\ (v'v'^*w'w'^*)+tr\ (v'v'^*w''w''^*)+
tr\ (v''v''^*w'w'^*)+tr\ v''v''^*w''w''^*)\\
&= |\langle v',w'\rangle|^2+|\langle v',w''\rangle|^2+
|\langle v'',w'\rangle|^2+|\langle v'',w''\rangle|^2\\
&= \|Pwv'\|^2+\|P_wv''\|^2\\
&= |\langle v,w\rangle|^2+|\langle iv,w\rangle|^2\\
&= 2|\langle v,w\rangle|^2.
\end{align*}
\end{proof}

\begin{corollary}
If $\{v_{ij}\}_{j=1}^n$ are mutually unbiased orthonormal bases 
for $\CC^n$, for $i=1,2,\ldots,k$, then the rank 2 projections
$\{P_{ij}\}_{i=1,j=1}^{\ n,\ k}$ onto 
$\{S_{v_{ij}}\}_{i=1,j=1}^{\ n,\ k}$ 
in $\RR^{2n}$ are mutually unbiased.
\end{corollary}

\begin{proof}
By Corollary \ref{cor}, for each $i$, $\{S_{v_{ij}}\}_{j=1}^n$ is an orthogonal family
of two dimensional subspaces in $\RR^{2n}$.  Hence, for $j\not= k$, 
\[ \langle P_{ij},P_{ik}\rangle = tr\ P_{ij}P_{ik}=0.\]
Also, by Proposition \ref{p23}, if $i\not= k$,
\[ \langle P_{ij},P_{k,l}\rangle = tr\ P_{ij}P_{kl} =
 =2|\langle v_{ij},v_{kl}\rangle|^2 = \frac{2}{n}.\]
\end{proof}

\section{Fusion Frames}

In this section, we will apply our association of vectors 
in $\CC^n$ to
rank two projections in $\RR^n$ to answer a longstanding problem
in fusion frame theory.  This topic was introduced in \cite{CK}.
Fusion frames have application to dimension reduction and
Grassmannian packings \cite{K}.

\begin{definition}
Given a family of subspaces $\{W_i\}_{i=1}^m$ with respective
projections $\{P_i\}_{i=1}^m$ in $\RR^n$ or $\CC^n$, and given
$a_i>0$ for $i=1,2,\ldots,m$, we say $(W_i,a_i)_{i=1}^m$ (respectively, $(P_i,a_i)_{i=1}^m$) is a {\bf fusion frame}
with fusion frame bounds $0<A\le B<\infty$ if for all vectors
$v$ we have:
\[ A \|v\|^2 \le \sum_{i=1}^ma_i^2\|P_iv\|^2 \le B\|v\|^2.
\]
\end{definition}

We start by showing that frames in $\CC^n$ will associate with
fusion frames of 2-dimensional subspaces of $\RR^{2n}$.  

\begin{theorem}
Let $\{v_i\}_{i=1}^m$ be a frame in $\CC^n$
and let $W_i=S_{v_i}$ with orthogonal projection
$P_i$ for $i\in [m]$. The
following are equivalent:
\begin{enumerate}
\item $\{v_i\}_{i=1}^m$ has frame bounds $A,B$.
\item $(W_i,\|v_i\|)_{i=1}^m$ is a fusion frame
of two dimensional subspaces for $\RR^{2n}$ with fusion frame bounds $A,B$.
\end{enumerate}
\end{theorem}

\begin{proof}
This is immediate since given $w\in \CC^n$,
\[ \sum_{i=1}^m|\langle w,v_i\rangle|^2=\sum_{i=1}^m\|v_i'\|^2\|P_iw'\|^2.\]
\end{proof}

It is exceptionally difficult to construct tight fusion
frames.  Especially since they often do not exist.  For
example, there does not exist a tight fusion frame for
$\RR^3$ consisting of two 2-dimensional subspaces.  To see
this, let $(W_i,a_i)_{i=1}^2$ be a fusion frame for $\RR^3$
with associated projections $\{P_i\}_{i=1}^2$.
Let $x \in W_1\cap W_2$.  Then
\[ a_1^2\|P_1x\|^2+a_2^2\|P_2x\|^2= (a_1^2+b_1^2)\|x\|^2.\]
But, if $x\in W_1$ but not in $
W_2$ then $\|P_2x\|^2 < \|x\|^2$ and so
\[ a_1^2\|P_1x\|^2+a_2^2\|P_2x\|^2< (a_1^2+b_1^2)\|x\|^2.\]
I.e.  This fusion frame is not tight.  It is believed that this
problem occurs over and over with odd numbers of two dimensional
subspaces of $\RR^3$ as well as occuring in $\RR^{2n+1}$ for
all n.
It has been a longstanding open problem whether
this is a problem of {\it odd dimensions} or does this
problem show up in even dimensions also.
I.e. Is there a tight fusion frame of m 2-dimensional subspaces in $\RR^{2n}$ for all n and all $m\ge n$.
We see that the above answers this in the affirmative.

\begin{corollary}
The following fusion frames exist:
\begin{enumerate}
\item For every $m\ge n$, there is a tight frame
 $\{v_i\}_{i=1}^m$ for $\CC^n$ and then
 for $W_i=S_{v_i}$, $\{W_i,\|v_i\|\}_{i=1}^m$
is a tight fusion frame of two dimensional subspaces of
$\RR^{2n}$.
\item  For every $m\ge n$, there is an equal norm
Parseval frame $\{v_i\}_{i=1}^m$ for $\CC^n$ (and so $\|v_i\|^2 =
\frac{n}{m}$) and then
for $W_i=S_{v_i}$, $\{W_i,\sqrt{\frac{n}{m}}\}_{i=1}^m$ is a tight
fusion frame of two dimensional subspaces of $\RR^{2n}$. 
\end{enumerate}
\end{corollary}

And for the general case we have,

\begin{theorem}
If $\{W_i,a_i\}_{i=1}^m$ be a fusion frame for
$\CC^n$ with dim $W_i=k_i$ for $i\in [m]$, with fusion frame bounds $A,B$ then
$\{V_i,a_i\}_{i=1}^m$, the induced $2k_i$-dimensional
subspaces of $\RR^{2n}$, form a fusion frame in 
$\RR^{2n}$ with fusion frame bounds $A,B$.   
\end{theorem}

\begin{proof}
Let $Q_i$ (respectively, $P_i$) be the projection onto $W_i$
(respectively, $V_i$).
By the proof of Proposition \ref{p7} we have for all $x\in \CC^n$:
$\|Q_ix\|=\|P_ix'\|$.  It follows that
\[ \sum_{i=1}^ma_i^2\|Q_ix\|^2 = \sum_{i=1}^ma_i^2\|P_ix'\|^2.\]
This proves the result.
\end{proof}

\section{Classifying Norm Retrieval}

We will give a theorem similar to the Edidin Theorem
but which classifies norm retrieval.

\begin{definition}
A family of projections $\{P_i\}_{i=1}^m$ on $\HH^n$ does {\bf norm
retrieval} if whenever $x,y\in \HH^n$ and $\|P_ix\|=\|P_iy\|$, for all
$i=1,2,\ldots,m$, we have that $\|x\|=\|y\|$.
\end{definition}

It is immediate that if $\{P_i\}_{i=1}^m$ does phase retrieval then
it does norm retrieval.  The converse fails since orthonormal
bases do norm retrieval and must fail phase retrieval.
The importance of norm retrieval is that \cite{CCPW} if a family of
projections $\{P_i\}_{i=1}^m$ does phase retrieval on $\HH^n$ then
$\{(I-P_i)\}_{i=1}^m$ does phase retrieval if and only if it does
norm retrieval.

\begin{theorem}
Given projections $\{P_i\}_{i=1}^m$ on $\RR^n$ the following are
equivalent:
\begin{enumerate}
\item $\{P_i\}_{i=1}^m$ does norm retrieval.
\item For every $0\not= x\in \RR^n$, we have
\[ \left [ span\ \{P_ix\}\right ]^{\perp}\subset x^{\perp}.\]
\item For every $0\not= x\in \RR^n$, we have that
$x \in span \ \{P_ix\}_{i=1}^m$.
\end{enumerate}
\end{theorem}

\begin{proof}
$(2) \Rightarrow (1)$:  We will prove the contrapositive.  If norm retrieval fails, then there
are vectors $x,y \in \RR^n$ with $\|P_ix\|=\|P_iy\|$ for all $i=12,\ldots,m$ but $\|x\|\not=\|y\|$.
This implies if $v=x+y$ and $w=x-y$ then $v,w$ are
not orthogonal. But 
$ w \in \left [span\ \{P_iv\}_{i=1}^m\right ]^{\perp},$
so property (2) fails.

$(2) \Leftrightarrow (3)$:  This is immediate.

$(2)\Rightarrow (1)$:  Again by the contrapositive,
there exists $v,w$ which are not orthogonal (so
$w\not= 0$) but $w \perp span\ \{P_iv\}_{i=1}^m$
and so $P_iw \perp P_iv$..
Write $v=x+y$ and $w=x-y$.  Since $v,w$ are not
orthogonal, $\|x\|\not= \|y\|$.  But, since
$w \perp P_iv$,
\begin{align*}
 \|P_i(v+w)\|^2&=\langle P_iv+P_iw,P_iv+P_iw\rangle\\
&=  \|P_iv\|^2+\|P_iw\|^2\\
 &= \langle P_iv-P_iw,P_iv-P_iw
\rangle = \|P_i(v-w)\|^2.
\end{align*}
So $\|P_ix\|=\|P_iy\|$ for every i, while $\|x\|\not= 
\|y\|$.  I.e.  $\{P_i\}_{i=1}^m$ fails norm retrieval.
\end{proof}

In general, it is difficult to show that projections $\{P_i\}_{i=1}^m$ do norm retrieval (especially if they fail phase retrieval) and
even more difficult to show that this passes to $\{(I-P_i)\}_{i=1}^m$.  It is known that, in general, $\{P_i\}_{i=1}^m$ may do norm
retrieval (even phase retrieval) while $\{(I-P_i)\}_{i=1}^m$ fails
norm retrieval \cite{X}.
We will now give a slightly weaker sufficient condition for
norm retrieval (respectively, phase retrieval) to pass from
$\{P_i\}_{i=1}^m$ to $\{(I-P_i)\}_{i=1}^m$.

\begin{theorem}\label{t1}
Let $\{P_i\}_{i=1}^m$ be projections on $\RR^n$.  The following
are equivalent:
\begin{enumerate}
\item $\{P_i\}_{i=1}^m$ does norm retrieval and for every $y\in 
\RR^n$ there are scalars $\sum_{i=1}^ma_i \not= 1$ so that
$y=\sum_{i=1}^ma_iP_iy.$
\item $\{(I-P_i)\}_{i=1}^m$ does norm retrieval and for every $y\in 
\RR^n$ there are scalars $\sum_{i=1}^ma_i \not= 1$ so that
$y=\sum_{i=1}^ma_i(I-P_i)y.$
\end{enumerate}
\end{theorem}

\begin{proof}
$(1)\Rightarrow (2)$:  Given $0\not= y \in \RR^n$, choose $\{a_i\}_{i=1}^m$ 
  to that $\sum_{i=1}^ma_iP_iy=y$ and $\sum_{i=1}^ma_i\not= 1$.
Then,
\[ \sum_{i=1}^ma_i(I-P_i)y= \left (\sum_{i=1}^ma_i\right )y
- \sum_{i=1}^ma_iP_iy= \left ( \sum_{i=1}^ma_i -1\right )y.\]
So 
\[ y = \sum_{i=1}^m\frac{a_i}{\sum_{i=1}^ma_i-1}(I-P_i)y.\]
Also,
\[ \sum_{i=1}^m\frac{a_i}{\sum_{i=1}^ma_i-1}= \frac{\sum_{i=1}^ma_i}{\sum_{i=1}^ma_i-1} \not= 1,\]
So $y\in span\ \{(I-P_i)y\}_{i=1}^m.$
I.e. $\{(I-P_i)\}_{i=1}^m$
does norm retrieval.

$(2)\Rightarrow (1)$:  By symmetry.
\end{proof}

\begin{theorem}
Let projections $\{P_i\}_{i=1}^m$ do phase retrieval in $\RR^n$.
The following holds:

 If for every $0\not= y$ there exist scalars $\{a_i\}_{i=1}^m$
so that $y=\sum_{i=1}^ma_iP_iy$ 
and $\sum_{i=1}^ma_i\not= 1$, then $\{(I-P_i)\}_{i=1}^m$ does
phase retrieval.
\end{theorem}

\begin{proof}
  By Theorem \ref{t1}, $\{(I-P_i)\}_{i=1}^m$ does norm retrieval
  and hence phase retrieval.
\end{proof}

\begin{corollary}
If $\{P_i\}_{i=1}^m$ does norm retrieval (respectively, phase
retrieval) on $\RR^n$ and for every $y\in \RR^n$ either there
exists $\sum_{i=1}^ma_i \not= 1$ and $y=\sum_{i=1}^ma_iP_iy$
or there exists $\sum_{i=1}^ma_i\not= 0$ and $\sum_{i=1}^ma_iP_iy=0$,
then $\{(I-P_i)\}_{i=1}^m$ does norm retrieval (respectively, 
phase retrieval.
\end{corollary}

\begin{proof}
If $\sum_{i=1}^ma_iP_iy=y$ and $\sum_{i=1}^ma_i\not= 1$ then
\[ \sum_{i=1}^ma_i(I-P_i)y=\left ( \sum_{i=1}^ma_i \right )y
- \sum_{i=1}^ma_iP_iy=\left ( \sum_{i=1}^ma_i -1\right )y.\]
So $y\in span\ \{(I-P_i)y\}_{i=1}^m$.  If $\sum_{i=1}^ma_i\not= 0$
and $\sum_{i=1}^ma_iP_iy=0$, then
\[  \sum_{i=1}^ma_i(I-P_i)y=\left ( \sum_{i=1}^ma_i \right )y
- \sum_{i=1}^ma_iP_iy=\left ( \sum_{i=1}^ma_i \right )y,\]
and so $y\in span \ \{(I-P_i)y\}_{i=1}^m$.
\end{proof}

\section{Equiangular and Biangular Frames}

In this section we will see how our association changes equiangular
and biangular tight frames in $\CC^n$ into equiangular and biangular tight
fusion frames in $\RR^{2n}$.  

\begin{definition}
A family of unit vectors $\{\phi_i\}_{i=1}^m$ in $\RR^n$ or $\CC^n$ is said to be {\it k-angular} if there are constants $\alpha_1
> \alpha_2 > \cdots > \alpha_k \ge 0$ so that
\[ \{|\langle \phi_i,\phi_j\rangle| :i\not= j\}|= \{\alpha_i:i=1,2,\ldots,k\}\]
Similarly, a family of subspaces $\{W_i\}_{i=1}^m$ with respective
projections $\{P_i\}_{i=1}^m$ is ${\it k-angular}$ if
\[ |\{\langle P_i,P_j\rangle|:i\not= j\}|= \{\alpha_i:i=1,2,\ldots,k\}.\]
If $k=1$, we call this {\it equiangular} and if $k=2$ we call this
{\it biangular}.
\end{definition}

Heath and Strohmer \cite{SH} (See also \cite{BH,Sus}) made a
detailed analysis of this class of frames and related this to
several areas of research.  Wooters and Fields \cite{woot} use them (not
under this name) to obtain {\em unbiased} measurements to determine the 
state (or density operator) of a quantum system.  
 Grassmannian frames also
arise in communication theory.  In \cite{S2} Grassmannian frames are used for
low-bit rate channel feedback in MIMO systems. It also arises in {\bf Welsh bound equality} sequences
\cite{Wa}.  

The main problem in this area of research is that very few equiangular/biangular tight frames are known.  It is known that
the number of equiangular lines in $\RR^n$ is less than or equal
to $n(n+1)/2$ \cite{del} but these bounds are rarely achieved.
For example, the maximal number of equiangular lines in $\RR^{d}$
is just 28 for all . $7\le d\le 14$ \cite{del}. In the complex case, the maximal number of
equiangular lines in $\CC^n$ is $n^2$ \cite{del}.  It is an open
problem whether this number is always attained.  We will now show
how to transfer equiangular lines from $\CC^n$ to equiangular
fusion frames of two dimensional subspaces of 
(respectively, equiangular families of rank 2 projections)
$\RR^{2n}$.  This provides many more equiangular sets to deal with
in $\RR^{2n}$ if researchers can live with rank 2 projections
instead of vectors.  For example, there are only 28 equiangular
lines in $\RR^{14}$, but there are 15 equiangular families of
rank 2 projections in $\RR^{14}$.

\begin{theorem}
If a unit norm tight frame $\{\phi_i\}_{i=1}^m$ in $\CC^n$ is
equiangular, then the rank 2 projections $P_i$ onto  $S_{v_i}$ 
in $\RR^{2n}$ form an equiangular tight fusion frame in $\RR^{2n}$.
\end{theorem} 

\begin{proof}
We just observe that 
\[ |\langle \phi_i,\phi_j\rangle|^2=2\|P_i\phi_j'\|^2
= tr\ P_iP_j=\langle P_i,P_j\rangle.\]
\end{proof}

Similarly we have:

\begin{theorem}
If a unit norm tight frame $\{\phi_i\}_{i=1}^m$ in $\CC^n$ is
k-angular, then the rank 2 projections $P_i$ onto  $S_{\phi_i}$ 
in $\RR^{2n}$ form a k-angular tight fusion frame in $\RR^{2n}$.
\end{theorem}

\end{document}